  \documentclass[11pt,a4paper,twoside]{article}

  \usepackage[margin=3.40cm]{geometry}
  \usepackage{fourier}
  \usepackage{amsmath,amssymb,amsthm,amsfonts}
  \usepackage{bm}
  \usepackage{mathabx}
  \usepackage{xspace,xcolor}
  \usepackage{parskip,fancyhdr,url}
  
  \usepackage{enumitem}
  \usepackage{todonotes}
  \usepackage[
     breaklinks,colorlinks,
     citecolor=blue,linkcolor=blue,urlcolor=teal
  ]{hyperref}
  \usepackage{tikz, tikz-cd}
  \usetikzlibrary{
    arrows, calc, matrix, decorations.pathmorphing, shapes, arrows
    }
  \usepackage{mathtools}

  \newcommand\Rsquigarrow[1]{%
\mathrel{%
\begin{tikzpicture}[baseline= {( $ (current bounding box.south) + (0,-0.5ex) $ )}]
  \node[inner sep=.5ex] (a) {$\scriptstyle #1$};
  \path[draw,implies-,double distance between line centers=1.5pt,decorate,
    decoration={snake,amplitude=0.7pt,segment length=1.2mm,pre=lineto,
    pre   length=4pt}]
    (a.south east) -- (a.south west);
\end{tikzpicture}}%
}

  \begingroup
    \makeatletter
    \@for\theoremstyle:=definition,remark,plain\do{%
      \expandafter\g@addto@macro\csname th@\theoremstyle\endcsname{%
        \addtolength\thm@preskip\parskip
        }%
      }
  \endgroup

  \makeatletter
    \def\tagform@#1{\maketag@@@{%
     \textbf{(\ignorespaces#1\unskip\@@italiccorr)}}}%
     \renewcommand{\eqref}[1]{\textup{\maketag@@@{(\ignorespaces%
          {\ref{#1}}\unskip\@@italiccorr)}}}
  \makeatother

  \newcommand\address[1]{}
  \newcommand\email[1]{}
  \newcommand\dedicatory[1]{}

  \theoremstyle{plain}
  \newtheorem{theorem}[equation]{Theorem}
  \newtheorem{proposition}[equation]{Proposition}
  \newtheorem{corollary}[equation]{Corollary}
  \newtheorem{lemma}[equation]{Lemma}

  \theoremstyle{definition}
  \newtheorem{definition}[equation]{Definition}

  \theoremstyle{axiom}

  \numberwithin{equation}{section}

  \theoremstyle{problem}

  \newcommand{\bsf}[1]{\ensuremath{\bm{\mathnormal{\mathsf{#1}}}}\xspace}

  \newcommand{\Born}{\ensuremath{\bsf{Coarse}\xspace}}
  \newcommand{\CMed}{\ensuremath{\bsf{CMed}\xspace}}

  \DeclareMathOperator{\Obj}{Obj} 
  \DeclareMathOperator{\Mor}{Mor} 

  \DeclareMathOperator{\Tuple}{Tuple} 

  \newcommand{\uc}{u.c.\xspace}

  \newcommand{\uclim}{\mathop{\ooalign{$\lim$\cr
  \hidewidth\raise-1.10ex\hbox{$\leftsquigarrow\mkern-0.5mu$}\cr}}}

  \DeclareMathOperator{\Rips}{Rips}

  \newcommand{\set}[1]{\ensuremath{\left\{ {#1} \right\}}\xspace}

  \newcommand{\st}{\ensuremath{\,\, \colon \,\,}\xspace}
  \newcommand{\from}{\ensuremath{\colon \thinspace}\xspace}
  
  \newcommand{\sqto}{\ensuremath{\rightsquigarrow}\xspace}
  \newcommand{\sqtto}{\ensuremath{\Rsquigarrow{\;\;\;\;}}\xspace}

        \newcommand{\C}{\ensuremath{\mathcal{C}}\xspace}
        
	\newcommand{\X}{\ensuremath{\mathcal{X}}\xspace}

	\newcommand{\calS}{\ensuremath{\mathcal{S}}\xspace}
       \newcommand{\calR}{\ensuremath{\mathcal{R}}\xspace}

  \newcommand{\calN}{\ensuremath{\mathcal{N}}\xspace}

  \newcommand{\sJ}{\mathsf{J}}

  \newcommand{\param}{{\mathchoice{\mkern1mu\mbox{\raise2.2pt\hbox{$
  \centerdot$}}
  \mkern1mu}{\mkern1mu\mbox{\raise2.2pt\hbox{$\centerdot$}}\mkern1mu}{
  \mkern1.5mu\centerdot\mkern1.5mu}{\mkern1.5mu\centerdot\mkern1.5mu}}}

  \fancypagestyle{plain}

  \AtEndDocument{\bigskip{\footnotesize%
  Department of Pure Mathematics, Xi'an Jiaotong--Liverpool University, 111 Ren'ai Road, Suzhou Industrial Park, Suzhou, Jiangsu, 215123, China
  \par \texttt{robert.tang@xjtlu.edu.cn}}}

\begin{document}

    \title{Coarse medians and universal quasigeodesic cones}
    \author{Robert Tang} \date{16 March 2026}

  \maketitle \thispagestyle{empty}

  \begin{abstract}

    We show that any universal quasigeodesic cone of uniformly coarse median spaces admits a canonical coarse median structure. As an application, we recover a result of Bowditch which states that any hierarchically hyperbolic space admits a coarse median structure compatible with the projections to the hyperbolic factor spaces.
  \end{abstract}


  \section{Introduction}

  Coarse median spaces and groups, introduced by Bowditch \cite{Bow13a}, are equipped with a ternary operator which captures useful features of nonpositive curvature. These generalise both Gromov hyperbolic spaces and CAT(0)--cube complexes, and include mapping class groups, Teichm\"uller space, and geometrically finite Kleinian groups among examples \cite{Bow16}. The coarse median viewpoint has been productive, providing a unified framework for questions about Dehn functions, asymptotic cones \cite{Bow13a}, quasi‑isometric rigidity \cite{Bow18}, automorphisms \cite{Fio24}, and analytic properties \cite{Zei16, SW17}.

  A natural coarse median on the mapping class groups arises from the Behrstock--Minsky centroid construction: subsurface projections to curve graphs yield centroids in each hyperbolic factor \cite{MM99,MM00}; these then assemble to a global centroid whose projections coarsely agree with the factorwise centroids \cite{BM11}. In \cite{Bow18}, Bowditch extends this approach to hierarchically hyperbolic spaces (HHS), as developed by Behrstock--Hagen--Sisto \cite{HHS1,HHS2}, showing that coarse medians on uniformly hyperbolic factor spaces induce a coarse median on the total space uniformly compatible with the projection data.

  In this note, we revisit these ``pulling back'' constructions from the perspective of universal quasigeodesic cones, as introduced in a recent paper of the author \cite{Tang-rips}. Informally, these act as universal quasigeodesic receivers over diagrams of metric spaces equipped with uniformly controlled maps between them (see Section \ref{sec:coarse} for the precise definition). We prove that if the factor spaces admit uniform coarse medians so that the given maps are uniformly coarse median preserving, then they induce a compatible coarse median on the universal quasigeodesic cone; see Theorem \ref{thm:uqc-median} for a precise statement.

  \begin{theorem}\label{thm:main}
    Let $D$ be a uniformly controlled diagram of uniformly coarse median spaces $(X_i, \mu_i)$ whose maps are uniformly coarse median preserving. Then any universal quasigeodesic uniformly controlled cone $X$ over $D$ admits a canonical (up to closeness) coarse median $\mu$ such that all legs $(X,\mu) \to (X_i, \mu_i)$ are uniformly coarse median preserving.
  \end{theorem}

  One main result from \cite{Tang-rips} is that any HHS realises a universal quasigeodesic cone over its diagram of hyperbolic factor spaces together with the usual pairwise constraint spaces. Theorem \ref{thm:main} applies to such diagrams and so we recover Bowditch’s coarse median theorem for HHSs (see Section \ref{sec:BCII}). More generally, this applies to families of uniformly coarse median spaces equipped with pairwise constraints either imposing no constraint or a bounded coarse interval image property. Theorem \ref{thm:main} further gives uniqueness of the coarse median on an HHS (up to closeness) compatible with given HHS--data; we remark that without prescribing HHS--data an HHS may admit (possibly uncountably) many distinct coarse medians \cite{Man-short}, though uniqueness does hold in some settings \cite{FS-unique}.

  The proof of Theorem \ref{thm:main} is constructive: by our earlier work \cite{Tang-rips}, a universal quasigeodesic cone can be realised using an explicit Rips--tuple space associated to the data; the desired coarse median is then induced by the product median. This highlights the utility of these tools for constructively promoting factorwise structures to global ones.

  \subsection*{Acknowledgements}

  The author thanks Mark Hagen, Robert Kropholler, Harry Petyt, Alessandro Sisto, Davide Spriano, and Jing Tao for interesting conversations.
  The author acknowledges support from the National Natural Science Foundation of China (NSFC 12101503); the Suzhou Science and Technology Development Planning Programme (ZXL2022473); and the XJTLU Research Development Fund (RDF-23-01-121).

  \section{Coarse geometry}\label{sec:coarse}

  Let $X$ and $Y$ be (extended) metric spaces. An \emph{upper control} for a map $f \from X \to Y$ is a proper increasing function $\rho \from [0,\infty) \to [0,\infty)$ such that such that $d_X(x,x') \leq t$ implies $d_Y(fx,fx') \leq \rho(t)$ for all $x,x' \in X$ and $t \geq 0$. We say $f$ is \emph{controlled} if it admits an upper control, and \emph{coarsely Lipschitz} if it admits an affine upper control. For any subset $U \subseteq Y$, let $\calN_r(U)$ denote its closed metric $r$--neighbourhood in $Y$. We call $f$ \emph{coarsely surjective} if $\calN_r f(X) = Y$ for some $r \geq 0$. For $\kappa \geq 0$ and points $y,y' \in Y$, write $y \approx_\kappa y'$ to mean $d_Y(y,y') \leq \kappa$. Say two maps $f,g \from X \to Y$ are \emph{$\kappa$--close}, denoted as $f \approx_\kappa g$, if $f(x) \approx_\kappa g(x)$ for all $x \in X$. We say $f,g$ are \emph{close}, and denote this by $f \approx g$, if $f \approx_\kappa g$ for some $\kappa \geq 0$. Note that closeness depends only on the codomain metric. A controlled map $f \from X \to Y$ is a \emph{coarse equivalence} if there exists a controlled map $g \from Y \to X$ such that $gf \approx 1_X$ and $fg \approx 1_Y$; such a map $g$ is called a \emph{coarse inverse} for $f$, and is unique up to closeness. If, in addition, $f$ and $g$ are coarsely Lipschitz then we call $f$ a \emph{quasi-isometry}.

  \subsection{Uniformly controlled cones}\label{sec:uccone}

  We recall the notions of uniformly controlled diagrams and cones from \cite{Tang-rips}. Let $\Born$ be the category with extended metric spaces asobjects, and controlled maps as morphisms. Let $\sJ$ be a directed multigraph, with vertex set $\Obj(\sJ)$ and edge set $\Mor(\sJ)$. A \emph{uniformly controlled (\uc) diagram} $D \from \sJ \sqto \Born$ (of shape $\sJ$)   comprises:
  \begin{itemize}
   \item a metric space $D_j$ for each $j \in \Obj(\sJ)$, and
   \item a controlled map $D_\phi \from D_i \to D_j$ for each arrow $\phi \from i \to j$ in $\Mor(\sJ)$,
  \end{itemize}
  such that all $D_\phi$ admit a common upper control.

  The ``squiggly arrow'' notation is intended to suggest an analogy with diagrams in the usual categorical sense, but with uniform control and error involved. We do not require existence of compositions or coarse commutativity for \uc~diagrams.

  Let $W$ be a metric space. A \emph{uniformly controlled cone} $\lambda \from W \sqtto D$ over $D$ with \emph{apex} $W$   comprises a uniformly controlled family of \emph{legs} $\lambda_j \from W \to D_j$ for $j \in \Obj(\sJ)$ such that
  \begin{center}
    \tikzcdset{arrow style=math font}
    \begin{tikzcd}[column sep=large]
    W \arrow[d, "\lambda_i"'] \arrow[dr, "\lambda_j"] & \\
    D_i \arrow[r, "D_\phi"'] & D_j
    \end{tikzcd}
    \end{center}
  uniformly coarsely commutes for all arrows $\phi \from i \to j$.   In other words, there exists $\kappa \geq 0$ such that $\lambda_j \approx_\kappa D_\phi \lambda_i$ for all arrows $\phi$. Two \uc~cones $\lambda, \lambda' \from W \sqtto D$ are \emph{close}, denoted as $\lambda \approx \lambda'$, if their legs are uniformly close. We call $\lambda \from W \sqtto D$ a \emph{\uc~limit cone} over $D$ if for any \uc~cone $\zeta \from Z \sqtto D$, there exists a unique (up to closeness) controlled map $f \from Z \to W$ such that $\zeta \approx \lambda f$. In other words, the diagram
    \begin{center}
    \tikzcdset{arrow style=math font}
    \begin{tikzcd}[column sep=large]
    Z \arrow[r, "f", dashed] \arrow[dr, "\zeta_j"'] & W \arrow[d, "\lambda_j"]   \\
    & D_j
    \end{tikzcd}
    \end{center}
  uniformly coarsely commutes for all legs. Any \uc~limit cone over $D$, if it exists, is unique up to coarse equivalence (unique up to closeness); we shall denote it by $\uclim_\sJ D$.

    For $\kappa \geq 0$, the \emph{$\kappa$--consistent tuplespace} of $D$ is
    \[\Tuple^\kappa(D) := \set{(x_j)_j \in \prod_j D_j \st x_j \approx_\kappa D_\phi x_j \textrm{ for all arrows } \phi} \subseteq \prod_j D_j\]
  equipped with the induced $\ell^\infty$--metric. The 1--Lipschitz projections $\pi^\kappa_j \from \Tuple^\kappa D \to D_j$ form the legs of a \uc~cone $\pi^\kappa \from \Tuple^\kappa(D) \sqtto D$.  The following is immediate.

   \begin{lemma}\label{lem:uc-factor}
   Let $\zeta \from Z \sqtto D$ be a \uc~cone. Then for all $\kappa$ large, $\zeta$ factors as $\zeta = \pi^\kappa f^\kappa$ where $f^\kappa \from Z \to \Tuple^\kappa D \subseteq \prod_j D_j$ is induced by $\prod_j \zeta_j$. \qed
  \end{lemma}

  The inclusions $\Tuple^\kappa(D) \hookrightarrow \Tuple^{\kappa'}(D)$ are isometric embedding for all $\kappa \leq \kappa'$; these give rise to a filtration $\Tuple^*(D)$. Say that $\Tuple^*(D)$ \emph{stabilises} if these inclusions are coarsely surjective for all $\kappa \leq \kappa'$ large.

  \begin{proposition}[Stable tuplespaces {\cite[Proposition 5.17]{Tang-rips}}]\label{prop:stable}
   Let $D \from \sJ \sqto \Born$ be a \uc~diagram. Then $D$ admits a \uc~limit cone if and only if $\Tuple^*(D)$ stabilises. In that case, $\pi^\kappa \from \Tuple^\kappa(D) \sqtto D$ realises the \uc~limit cone over $D$ for $\kappa$ large. \qed
  \end{proposition}

  \subsection{Universal quasigeodesic cones}\label{sec:uqc}

  We now turn our attention to \emph{quasigeodesic} \uc~cones, that is, where the apex is quasigeodesic. Recall that a metric space is \emph{quasigeodesic} (resp.~\emph{coarsely geodesic}) if it is quasi-isometric (resp.~coarsely equivalent) to a graph with the standard graph (i.e.~combinatorial) metric; this is equivalent to the usual definition in terms of quasi- (resp.~coarse) geodesics \cite[Proposition 3.20]{Tang-rips}. Any controlled map with quasigeodesic domain is necessarily coarsely Lipschitz. In particular, the legs of any quasigeodesic \uc~cone are uniformly coarsely Lipschitz. Say that $\lambda \from W \sqtto D$ is a \emph{universal quasigeodesic \uc~cone} if every quasigeodesic \uc~cone $\zeta \from Z \sqtto D$ factors (up to closeness) through $\lambda$ via a unique (up to closeness) coarsely Lipschitz map. Any such cone is unique up to quasi-isometry.

  To compute universal quasigeodesic \uc~cones, we employ the Rips filtration.   Given a metric space $X$, the \emph{Rips graph} $\Rips_\sigma X$ at scale $\sigma \geq 0$ is the graph with vertex set $X$, with points $x,x'\in X$ declared adjacent if and only if $d_X(x,x') \leq \sigma$. We shall regard $\Rips_\sigma X$ as a metric space with underlying set $X$, where the metric is the restriction of the standard graph metric. The underlying identity on $X$ induces a
  \begin{itemize}
   \item Lipschitz map $\xi_\sigma \from \Rips_\sigma X \to X$ for all $\sigma \geq 0$, and
   \item a 1--Lipschitz map $\Rips_\sigma X \to \Rips_\tau X$ for all $\sigma \leq \tau$.
  \end{itemize}
  By considering the Rips graph at all scales, we obtain the \emph{(metric) Rips filtration $\Rips_*X$}. Say that $\Rips_*X$ stabilises if the maps $\Rips_\sigma X \to \Rips_\tau X$ are coarse equivalences (or, equivalently, quasi-isometries) for all $\sigma \leq \tau$ large. The next two results are standard.

  \begin{lemma}[{\cite[Lemma 3.22]{Tang-rips}}]
  \label{lem:rips-stable}
   A metric space $X$ is coarsely geodesic if and only if $\Rips_*X$ stabilises; in which case $\xi_\sigma \from \Rips_\sigma X \to X$ is a coarse equivalence for $\sigma$ large. \qed
   \end{lemma}

  \begin{proposition}[{\cite[Lemma 3.3 and Proposition 3.20]{Tang-rips}}]
  \label{prop:qg-rips}
    Let $X$ be a quasigeodesic space and $f \from X \to Y$ be a controlled map. Then for $\sigma \geq 0$ large, the map $f_\sigma \from X \to \Rips_\sigma Y$ induced by $f$ is controlled, hence coarsely Lipschitz. \qed
    \end{proposition}

  By applying the Rips filtration to the coarse tuplespaces, we obtain an explicit method for computing universal quasigeodesic \uc~cones.

  \begin{theorem}[Rips--tuple recipe {\cite[Theorem 5.19]{Tang-rips}}]\label{thm:rips-tuple}
   Let $D \from \sJ \sqto \Born$ be a \uc~diagram. The following are equivalent.
   \begin{enumerate}
    \item $D$ admits a universal quasigeodesic \uc~cone,
    \item $\pi^\kappa \xi_\sigma \from \Rips_\sigma \Tuple^\kappa D \sqtto D$ is a universal quasigeodesic cone for some $\sigma, \kappa \geq 0$,
    \item $\Tuple_* D$ stabilises to a coarsely geodesic space. \qed        \end{enumerate}
  \end{theorem}

  \subsection{Coarse median spaces}

   Background on coarse medians spaces can be found in \cite{Bow13a, Bow13b, Bow19}. We shall use an equivalent formulation due to Niblo--Wright--Zhang \cite{NWZ19}.

  \begin{definition}\label{def:median}
   A controlled map $\mu \from X^3 \to X$ is a \emph{coarse median} on a metric space $X$ if there exists $C \geq 0$ such that
   \begin{enumerate}
    \item (Coarse symmetry) $\mu(x_{\tau 1}, x_{\tau 2}, x_{\tau 3}) \approx_C \mu(x_1, x_2, x_3)$ for all permutations $\tau \in S_3$,
    \item (Coarse localisation) $\mu(x_1, x_1, x_2) \approx_C x_1$
    \item (Coarse four-point condition) $\mu(\mu(x_1, w, x_2), w, x_3) \approx_C \mu(x_1, w, \mu(x_2, w, x_3))$
   \end{enumerate}
   for all $x_1, x_2, x_3, w \in X$. Call $(X,\mu)$ a \emph{coarse median space} with \emph{coarse median constant} $C$.
  \end{definition}

    Here, we work with the $\ell^\infty$--metric on $X^3$.
    Any function $\mu' \from X^3 \to X$ close to a coarse median $\mu$ on $X$ is also a coarse median. We say two coarse medians $\mu, \mu'$ on $X$ are \emph{close} if they are close as controlled maps $\mu, \mu' \from X^3 \to X$. Let $[\mu]$ denote the closeness class of a coarse median $\mu$.

   A controlled map $f \from (X,\mu) \to (Y,\nu)$ between coarse median spaces is \emph{coarse-median preserving} (CMP) if $f \mu \approx \nu \circ f^{\times 3} \from X^3 \to Y$. In other words, there exists $\kappa\geq 0$ such that $f(\mu(x,y,z)) \approx_\kappa \nu(fx, fy, fz)$ for all $x,y,z\in X$. The assumption that the involved maps are controlled ensures that the CMP property is preserved under composition.

   \begin{lemma}\label{lem:ce-median}
   Let $(X,\mu)$ be a coarse median space and suppose that $f \from W \to X$ is a coarse equivalence. Then there exists a unique (up to closeness) coarse median $\nu$ on $W$ such that $f \from (W,\nu) \to (X, \mu)$ is CMP.    \end{lemma}

   \proof
   Let $g \from X \to W$ be a coarse inverse of $f$. It is straightforward to verify that the  ternary operator $\nu:= f \circ \mu \circ g^{\times 3} \from Y^3 \to Y$ defines a coarse median on $Y$. Note that
   \[\nu \circ f^{\times 3} \approx f \circ \mu \circ g^{\times 3}\circ f^{\times 3} \approx f \mu \circ 1_W^{\times 3} \approx f\mu,\]
   hence $f \from (W,\nu) \to (X,\mu)$ is CMP. To verify uniqueness, suppose $\nu'$ is coarse median on $Y$ such that $f\from (W,\nu') \to (X,\mu)$ is CMP. Then
    \[\nu \approx \nu \circ f^{\times 3} \circ g^{\times 3} \approx \nu' \circ f^{\times 3} \circ g^{\times 3} \approx \nu'\]
    as desired.
   \endproof

    Let $(X,\mu)$ be a coarse median space with coarse median constant $C\geq 0$ and upper control $\rho$ for $\mu$. Given points $x,y \in X$ and $L \geq 0$, define their \emph{$L$--coarse interval} to be
    \[[x,y]_L := \set{z \in X \st \mu(x,y,z) \approx_L z}.\]

  \begin{lemma}\label{lem:intervals}
   Let $x,y \in X$. Then
   \begin{enumerate}
    \item $x, y \in [x,y]_{2C}$ and $\mu(x,y,z) \in [x,y]_{C + \rho C}$ for all $z\in X$,
    \item $\calN_r[x,y]_L \subseteq [x,y]_{L + r + \rho r}$ for all $L, r \geq 0$, and
    \item If $z \in [x,y]_L$ then $[x,z]_L \subseteq [x,y]_{L'}$ where $L'$ depends only on $C$, $L$, and $\rho$.
   \end{enumerate}
  \end{lemma}

  \proof
  \begin{enumerate}
   \item
   Coarse symmetry and coarse localisation imply that $x, y \in [x,y]_{2C}$, while
   \[\mu(x,y, \mu(x,y,z)) \approx_C \mu(\mu(x,y,x), y, z) \approx_{\rho C}
\mu(x,y,z)\]
   follows from coarse symmetry and the coarse 4--point condition.
   \item Let $w \approx_r z \in [x,y]_L$. Then
   $\mu(x,y,w) \approx_{\rho r} \mu(x,y,z) \approx_L z \approx_r w$.
   \item Let $w \in [x,z]_L$. Then
   \begin{align*}
    w &\approx_{L+C} \mu(w,x,z) \approx_{\rho L} \mu(w,x,\mu(x,y,z)) \approx_{\rho C} \mu(w,x,\mu(z,x,y))\\
    &\approx_{C} \mu(\mu(w,x,z), x, y) \approx_{\rho (C+L)} \mu(w,x,y) \approx_C \mu(x,y,w). \qedhere
   \end{align*}
  \end{enumerate}
   \endproof

  Next, we recall a coarse version of the 5--point condition from median algebras. This follows from Bowditch's original definition of coarse medians; see \cite[Section 2.3]{NWZ21}.

  \begin{lemma}\label{lem:five}
   There exists a constant $E\geq 0$ depending on $C$ and $\rho$ such that
   \[\mu(v,w, \mu(x,y,z)) \approx_E \mu(\mu(v,w,x), \mu(v,w,y), z)\]
   for all $v,w,x,y,z \in X$. \qed
  \end{lemma}

  \begin{lemma}\label{lem:tripod}
  If $O \in [x,y]_L \cap [y,z]_L \cap [z,x]_L$ then $O \approx_R \mu(x,y,z)$ where $R$ depends on $\rho$, $L$, and $C$.
  \end{lemma}

  \proof
  Let $m = \mu(x,y,z)$. By the coarse 5--point condition, we deduce
   \[
   \mu(O, m, x) \approx_{\rho L} \mu(\mu(x,y,O), \mu(x,y,z), x)
   \approx_{E} \mu(x, y, \mu(O, z, x))
   \approx_{\rho (L+C)} \mu(x, y, O)
    \approx_{L} O.\]
    Choosing $K' \geq \rho L + E + \rho(L+C) + L$, it follows that
  \begin{align*}
   m & \approx_C \mu(m, m, O) \approx_{\rho (C + C')} \mu(\mu(y,z,x), \mu(y,z,m), O)\\
   &\approx_{E} \mu(y, z, \mu(x, m, O)) \approx_{\rho (K'+C)} \mu(y, z, O) \approx_{L} O. \qedhere
    \end{align*}
   \endproof

Following Fioravanti \cite{Fio24}, a subset $U \subseteq X$ is an \emph{approximate median subalgebra} if there exists $R \geq 0$ such that $\mu(U^3) \subseteq \calN_R(U)$. For such a subset, define a ternary operator $\mu|_U$ on $U$ by choosing $\mu|_U(x,y,z) \in U$ to be any point within distance $R$ of $\mu(x,y,z)$ in $X$. Up to closeness, $\mu|_U$ is the unique coarse median on $U$ which is close to the restriction of $\mu$ to $U^3 \subseteq X^3$; in particular, $[\mu|_U]$ depends only on $[\mu]$. We say that $\mu|_U$ is \emph{induced} by $\mu$. Any subset $U' \subseteq X$ at finite Hausdorff distance from $U$ is also an approximate median subalgebra.

    \section{Coarse median cones}

    We now consider \uc~diagrams and cones in the coarse median setting. The \emph{coarse median category} $\CMed$ has coarse median spaces as objects and CMP controlled maps as morphisms. A \emph{\uc~coarse median diagram} $M \from \sJ \sqto \CMed$ is a \uc~diagram with uniformly coarse median spaces $M_j := (X_j, \mu_j)$ as objects (i.e., they admit a common coarse median constant and upper control), and whose bonding maps $M_\phi \from (X_i, \mu_i) \to (X_j, \mu_j)$ are uniformly CMP. Let $\hat M \from \sJ \sqto \Born$ denote the underlying \uc~diagram of metric spaces, obtained by forgetting the coarse medians. A \emph{\uc~coarse median cone} $\lambda \from (W,\nu) \sqtto M$ is a \uc~cone over $\hat M$ with a coarse median apex $(W,\nu)$ whose legs are uniformly CMP. Call $\lambda$ a \emph{universal} \uc~coarse median cone if every \uc~coarse median cone $\zeta \from (Z,\eta) \sqtto M$ factors through $\lambda$ via a unique (up to closeness) CMP controlled map.

   Fix a common coarse median constant $C\geq 0$ and upper control $\rho$ for the objects $(M_j, \mu_j)$, and $c \geq 0$ such that $\mu_j \circ (M_\phi)^{\times 3} \approx_c (M_\phi) \circ \mu_i$ for all arrows $\phi \from i \to j$. The uniformity assumptions allow us to equip the $\ell^\infty$--product $\prod_j X_j$ with the product coarse median $\prod_j \mu_j$; this coarse median space shall be denoted as $\prod_j M_j$.

   Say that a subset $U \subseteq \prod_j M_j$ is \emph{$M$--compatible} if $U \subset \Tuple^\kappa \hat M$ for some $\kappa \geq 0$. For such a set $U$, we may define a \uc~cone $\pi|_U \from U \sqtto \hat M$ where the legs $(\pi|_U)_j$ are restrictions of the projection to each $M_j$--factor. If, in addition, $U$ is an approximate median subalgebra of $\prod_j M_j$ then $\pi|_U \from (U,(\prod_j \mu_j)|_U) \sqtto M$ is a \uc~coarse median cone. (Indeed, each factorwise projection $\pi_i \from \prod_j M_j \to M_i$ is uniformly CMP.) Note that if $U$ is $M$--compatible (resp.~an approximate median subalgebra) then so is $\calN_r(U)$ for any $r \geq 0$.

   We show that any \uc~coarse median cone $\lambda \from (W,\nu) \sqtto M$ factors canonically (up to closeness) through an $M$--compatible approximate median subalgebra. Let $I(\lambda) \subseteq \prod_j M_j$ denote the image of $\prod_j \lambda_j \from (W,\nu) \to \prod_j M_j$.

   \begin{lemma}\label{lem:prod-med}
    Let $\lambda \from (W,\nu) \sqtto M$ be a \uc~coarse median cone. Then $\prod_j \lambda_j$ is CMP and $I(\lambda)$ is an $M$--compatible approximate median subalgebra.
   \end{lemma}

   \proof
   By assumption, there exists $r \geq 0$ such that $\lambda_j \nu \approx_r \mu_j \lambda_j^{\times 3}$ for all objects $j$, and so    \[(\prod_j \lambda_j) \circ \nu \approx_r (\prod_j \mu_j)\circ(\prod_j \lambda_j)^{\times 3}.\]
   Note that $I(\lambda)^3 = (\prod_j \lambda_j)^{\times 3}(W^3)$. Thus, $(\prod_j \mu_j)(I(\lambda)^3) \subseteq \calN_r(I(\lambda))$, hence $I(\lambda)$ is an approximate median subalgebra. By Lemma \ref{lem:uc-factor}, $I(\lambda)$ is $M$--compatible.
  \endproof

   Therefore $\lambda = \pi|_{I(\lambda)} \circ (\prod_j \lambda_j)$ yields the desired factorisation, upon equipping $I(\lambda)$ with the coarse median induced by $\prod_j \mu_j$. Thus, to characterise universal \uc~coarse median cones over $M$, we may focus our attention on $M$--compatible approximate median subalgebras of $\prod_j M_j$. Define an ordering on subsets $U, U' \subseteq \prod_j M_j$ by declaring $U \prec U'$ if and only if $U \subseteq \calN_r(U')$ for some $r \geq 0$.

   \begin{proposition}\label{prop:maximal}
     A \uc~coarse median diagram $M\from \sJ \sqto \CMed$ admits a universal \uc~coarse median cone if and only if there exists a $\prec$--greatest $M$--compatible approximate median subalgebra $U$ of $\prod_j M_j$. In that case, a universal \uc~coarse median cone is realised by $\pi|_U$.
   \end{proposition}

   \proof
   Assume $\lambda \from (W,\nu) \sqtto M$ is a universal \uc~coarse median cone. By Lemma \ref{lem:prod-med}, $I(\lambda)$ is an $M$--compatible approximate median subalgebra of $\prod_j M_j$. Let $U \subseteq \prod_j M_j$ be any $M$--compatible approximate median subalgebra. Then $\pi|_U \from (U, (\prod_j \mu_j)|_U) \sqtto M$ is a \uc~coarse median cone. By the universal property, $\pi|_U$ factors through $\lambda$ up to closeness, and so $I(\pi|_U) = U$ lies in a bounded neighbourhood of $I(\lambda)$, hence $U \prec I(\lambda)$.

   For the converse, assume $U \subseteq \prod_j M_j$ is $\prec$--greatest among $M$--compatible approximate median subalgebras. For any \uc~coarse median cone $\zeta \from (Z,\nu) \sqtto M$, the set $I(\zeta)$ is an $M$--compatible approximate median subalgebra, by Lemma \ref{lem:prod-med}, hence $I(\zeta) \subseteq \calN_r(U)$ for some $r \geq 0$. Let $h$ be a coarse inverse to the inclusion $U \hookrightarrow \calN_r(U)$. Note that $h$ is CMP with respect to the induced coarse medians. Then $\zeta$ factors as
   \[\zeta = \pi|_{\calN_r(U)} \circ \prod_j \zeta_j \approx \pi|_U \circ h \circ \prod_j \zeta_j. \]
   To verify uniqueness (up to closeness), suppose that $\pi|_U \circ h' \approx_s \pi|_U \circ h''$ for some CMP controlled maps $h', h'' \from (Z, \nu) \to (U,(\prod_j \mu_j)|_U)$ and $s \geq 0$. Then
  \[h'z = (\pi_j h'z)_j = ((\pi|_U)_j h'z)_j \approx_s ((\pi|_U)_j h''z)_j = (\pi_j h''z)_j = h''z\]
  for all $z \in Z$, hence $h' \approx h''$. Therefore, $\pi|_U$ is a universal \uc~coarse median cone.
   \endproof

  We now show that $M$ induces a coarse median on its \uc~limit cone (if it exists).

   \begin{theorem}\label{thm:univ-cmed-cone}
   Let $M\from \sJ \sqto \CMed$ be a \uc~coarse median diagram such that $\uclim_\sJ \hat M$ exists. Then $\Tuple^\kappa \hat M$ is a $\prec$--greatest $M$--compatible approximate median subalgebra of $\prod_j M_j$ for $\kappa$ large. Consequently, a universal \uc~coarse median cone over $M$ is realised by $\Tuple^\kappa \hat M$ equipped with the coarse median induced by $\prod_j \mu_j$ for $\kappa$ large.
   \end{theorem}

   \proof
   By Proposition \ref{prop:stable}, $\Tuple^* \hat M$ stabilises at some threshold $\kappa_0 \geq 0$, and so
   \[\Tuple^{\kappa}(\hat M) \prec \Tuple^{\kappa'}(\hat M) \prec \Tuple^{\kappa_0}(\hat M) \prec \Tuple^\kappa(\hat M)\]
   for all $\kappa' \geq \kappa \geq \kappa_0$. We claim that $\Tuple^\kappa(\hat M)$ is an approximate median subalgebra of $\prod_j M_j$ for $\kappa \geq \kappa_0$. Let $x,y,z \in \Tuple^\kappa(\hat M)$. Then
      \[M_\phi \mu_i (x_i, y_i, z_i) \approx_C \mu_j (M_\phi x_i, M_\phi y_i, M_\phi z_i) \approx_{\rho \kappa} \mu_j (x_j, y_j, z_j)\]
   for all arrows $\phi \from i \to j$. Therefore $(\prod_j \mu_j)(x,y,z) \in \Tuple^{C + \rho\kappa}(\hat M) \prec \Tuple^\kappa(\hat M)$,
  yielding the claim. Note that any $M$--compatible set $U$ satisfies $U \prec \Tuple^\kappa(\hat M)$ for some $\kappa \geq \kappa_0$. It follows that $\Tuple^\kappa(\hat M)$ satisfies the desired $\prec$--greatest condition, and hence yields a universal \uc~coarse median cone over $M$ by Proposition \ref{prop:maximal}.
   \endproof

   \section{Quasigeodesic coarse median cones}

   We now turn our attention to \emph{quasigeodesic} \uc~coarse median cones, that is, where the apex is quasigeodesic. Call $\lambda \from (W,\mu) \sqtto M$ a \emph{universal} quasigeodesic \uc~coarse median cone if every quasigeodesic \uc~coarse median cone $\zeta\from (Z,\nu) \sqtto M$ factors through $\lambda$ via a unique (up to closeness) CMP controlled map. We shall consider the interaction between coarse medians and the Rips filtration.

   Given a coarse median $\mu$ on $X$, let $\psi_\sigma \from (\Rips_\sigma X)^3 \to \Rips_{\sigma} X$ be the ternary operator coinciding with $\mu$ on underlying sets.

   \begin{lemma}\label{lem:Rips-Lip}
    Assume that $\mu$ has upper control $\rho$. Then the function $(\Rips_\sigma X)^3 \to \Rips_{\rho \sigma} X$ coinciding with $\mu$ on underlying sets is 1--Lipschitz for all $\sigma \geq 0$.
   \end{lemma}

   \proof
   Observe that $(\Rips_\sigma X)^3$ is isometric to $\Rips_\sigma X^3$ via the underlying identity. Suppose $(x,y,z), (x',y',z') \in X^3$ are adjacent in $\Rips_\sigma X^3$. Then $d_{X^3}((x,y,z), (x',y',z')) \leq \sigma$, hence $d_X(\mu(x,y,z), \mu(x',y',z')) \leq \rho \sigma$. The result follows.
         \endproof

   \begin{lemma}\label{lem:Rips-median}
    Let $(X,\mu)$ be coarsely geodesic coarse median space. Then for $\sigma \geq 0$ large, $\psi_\sigma$ is a coarse median on $\Rips_\sigma X$. Moreover, $\xi_\sigma \from (\Rips_\sigma X, \psi_\sigma) \to (X, \mu)$ is CMP.
    \end{lemma}

   \proof
   Let $C$ be a coarse median constant for $(X,\mu)$ and $\rho$ be an upper control for $\mu$. We may assume, without loss of generality, that $\rho$ satisfies $\rho(t) \geq t$ for all $t \geq 0$. By Lemma \ref{lem:rips-stable}, there exists $\sigma_0 \geq 0$ such that the underlying identity $\Rips_{\rho\sigma} \to \Rips_\sigma X$ is a coarse equivalence for all $\sigma \geq \sigma_0$. Therefore, by Lemma \ref{lem:Rips-Lip}, $\psi_\sigma$ factors via controlled maps $(\Rips_\sigma X)^3 \to \Rips_{\rho\sigma}X \to \Rips_\sigma X$ for $\sigma \geq \sigma_0$, and is hence controlled. By choosing $\sigma \geq C$,    we obtain a coarse median space $(\Rips_\sigma X, \psi_\sigma)$ with coarse median constant 1. Note that $\mu \xi_\sigma^{\times 3} = \xi_\sigma \psi_\sigma$ by construction, and so $\xi_\sigma$ is CMP.
   \endproof

   The following result is a partial analogue of Proposition \ref{prop:maximal}

   \begin{proposition}\label{prop:qgeod-median}
   Let $M\from \sJ \sqto \CMed$ be a \uc~coarse median diagram. Suppose that there exists a $\prec$--greatest element $U$ among the coarsely geodesic $M$--compatible approximate median subalgebras of $\prod_j M_j$. Then $\Rips_\sigma U$ realises a universal quasigeodesic \uc~coarse median cone for $\sigma$ large.
   \end{proposition}

   \proof
   Assume $U \subseteq \prod_j M_j$ satisfies the given $\prec$--greatest property. Equip $U$ with the coarse median $\mu'$ induced by $\prod_j \mu_j$. Since $U$ is coarsely geodesic, by Lemma \ref{lem:rips-stable},    the map $\xi_\sigma \from \Rips_\sigma U \to U$ coinciding with the underlying identity is a coarse equivalence for $\sigma$ large. We claim that $\pi|_U \circ \xi_\sigma \from \Rips_\sigma U \sqtto M$ is a universal quasigeodesic \uc~coarse median cone. Let $\zeta \from (Z, \nu) \sqtto M$ be a quasigeodesic \uc~coarse median cone. Using the same argument as in the proof of Proposition \ref{prop:maximal}, we deduce that $\zeta$ factors through $\pi|_U$ via a unique (up to closeness) CMP controlled map $h \from (Z, \nu) \to (U, \mu')$.
   \begin{center}
    \tikzcdset{arrow style=math font}
    \begin{tikzcd}[row sep = large, column sep = huge]
    Z \arrow[dr, "h", dashed] \arrow[d, "h'"', dashed] \arrow[drr, "\zeta_j", bend left = 15] & & \\
    \Rips_\sigma U \arrow[r, "\xi_\sigma"] \arrow[rr, "(\pi|_{U})_j \circ \xi_\sigma"', bend right = 12] & U \arrow[r, "(\pi|_{U})_j"]  & M_j
    \end{tikzcd}
    \end{center}
    Since $Z$ is quasigeodesic, Proposition \ref{prop:qg-rips} asserts the existence of a coarsely Lipschitz map $h' \from Z \to \Rips_\sigma U$ such that $h \approx \xi_\sigma h'$. By Lemma \ref{lem:Rips-median}, the coarse equivalence $\xi_\sigma$ is CMP, and so $h'$ is CMP. Therefore, $\zeta \approx \pi|_U \circ h \approx (\pi|_U \circ \xi_\sigma) h'$ yields a desired factorisation. To verify uniqueness, suppose that $\zeta \approx (\pi|_U \circ \xi_\sigma) h''$ for some controlled map $h'' \from Z \to \Rips_\sigma U$. Then $\pi|_U \circ (\xi_\sigma h'') \approx \pi|_U \circ (\xi_\sigma h')$ and so, by arguing as in the proof of Proposition \ref{prop:maximal}, we deduce that $\xi_\sigma h'' \approx \xi_\sigma h'$. Since $\xi_\sigma$ is a coarse equivalence, it follows that $h'' \approx h'$. Thus, $\pi|_U \circ \xi_\sigma \from \Rips_\sigma U \sqtto M$ is a universal quasigeodesic \uc~coarse median cone.
   \endproof

   We finally prove Theorem \ref{thm:main}.

   \begin{theorem}\label{thm:uqc-median}
   Let $M\from \sJ \sqto \CMed$ be a \uc~coarse median diagram. Assume that its underlying \uc~diagram $\hat M$ admits a universal quasigeodesic \uc~cone $\lambda \from W \sqtto \hat M$. Then $W$ admits a canonical (up to closeness) coarse median $\nu$ such that $\lambda \from (W, \nu) \sqtto M$ is a universal quasigeodesic \uc~coarse median cone.
   \end{theorem}

   \proof
    By Theorem \ref{thm:rips-tuple}, $\uclim_\sJ \hat M$ exists and is coarsely geodesic. Moreover, $\Tuple^* \hat M$ stabilises to a coarsely geodesic space. By choosing $\kappa \geq 0$ sufficiently large and appealing to Theorem \ref{thm:univ-cmed-cone}, we deduce that $\Tuple^\kappa \hat M$ is $\prec$--greatest among coarsely geodesic $M$--compatible approximate median subalgebras of $\prod_j M_j$. Therefore, by Proposition \ref{prop:qgeod-median}, we obtain a universal quasigeodesic \uc~coarse median cone $\zeta \from \Rips_\sigma \Tuple^\kappa \hat M \sqtto M$ for $\sigma$ large, where the legs are restrictions of factorwise projections on underlying sets. Note that $\zeta$ is also a \uc~limit cone (since $\xi_\sigma$ is a coarse equivalence), and so there exists a unique (up to closeness) coarse equivalence $f \from W \to \Rips_\sigma \Tuple^\kappa \hat M$ such that $\lambda \approx \zeta f$. In fact, $f$ is a quasi-isometry since its domain and codomain are quasigeodesic. By Lemma \ref{lem:ce-median}, we obtain a unique (up to closeness) coarse median $\nu$ on $W$ for which $f$ is CMP.
   \endproof

  \section{Bounded coarse interval image}\label{sec:BCII}

  In this section, we apply Theorem \ref{thm:uqc-median} to recover a theorem of Bowditch which states that any HHS admits a coarse median structure compatible with the coarse medians on the hyperbolic factor spaces \cite{Bow18}; in particular, this yields a coarse median for the mapping class group. We shall not require the definition of HHSs here; instead, the interested reader may refer to \cite{HHS1,HHS2,HHSsurvey} for further background. Our argument applies more generally to families of uniformly coarse median spaces satisfying some pairwise constraints modelled on the HHS--consistency axioms. We do not assume any quasigeodesicity (or coarse geodesicity) of the involved coarse median spaces.

  Consider a family $\{(\C U, \mu_U)\}_{U \in \calS}$ of uniformly coarse median spaces, with index set $\calS$. Here, we assume that all $(\C U, \mu_U)$ admit a common coarse median constant $C$ and upper control $\rho$. The set $\calS$ is equipped with a symmetric non-reflexive binary relation $\perp$ on $\calS$ called \emph{orthogonality} which we shall use to impose pairwise constraints on this family of spaces. For $\C U$ and $\C V$ where $U \perp V$, we impose no constraint; whereas for $U \not \perp V$ (and distinct), we impose the following:

  \textbf{Bounded coarse interval image.}
  There is a constant $B\geq C + \rho C$ such that for all $U \not \perp V$ distinct, (up to swapping $U$ and $V$) there exists a function $\theta^V_U \from \C V \to \C U$ and a point $O^U_V\in \C V$ such that for all $x_V, y_V \in \C V$, we have that $O^U_V \in [x_V, y_V]_B$ or $\theta^V_U x_V \approx_B \theta^V_U y_V$.

  This condition is modelled on the bounded geodesic image property in the setting where $\theta^V_U \from \C V \to \C U$ is the subsurface projection between curve graphs \cite{MM00}. In that case, $\C V$ is Gromov hyperbolic and so the condition amounts to saying that if $O^U_V$ does not lie near any geodesic from $x_V$ to $y_V$ in $\C V$ then the images of $x_V$ and $y_V$ under $\theta^V_U$ are uniformly close in $\C U$.

  Now, we define pairwise constraints on the family $\{(\C U, \mu_U)\}_{U \in \calS}$. Fix $K \geq B$. Whenever $U \perp V$, declare $\calR_{UV} := \C U \times \C V$. For $U \not\perp V$, declare
  \[\calR_{UV} := \set{(x_U, x_V) \in \C U \times \C V \st x_U \approx_K \theta^V_U x_V \;\; \textrm{or} \;\; x_V \approx_K O^U_V}. \]
  In either case, we equip $\calR_{UV}$ with the induced $\ell^\infty$--metric. These constraints simultaneously capture the consistency conditions arising from the bounded geodesic image property and the Behrstock inequality (by setting $\theta^V_U$ constant) from the HHS Axioms. Let us verify that these $\calR_{UV} \subseteq \C U \times \C V$ are uniform approximate median subalgebras.

  \begin{lemma}\label{lem:pairwise-median}
  There exists a constant $R \geq 0$ depending on $C$, $B$, $K$, and $\rho$ such for all distinct $U,V \in \calS$, we have that $(\mu_U \times \mu_V)(\calR_{UV}^3) \subseteq \calN_R(\calR_{UV})$ .
  \end{lemma}

  \proof
  The $U \perp V$ case is trivial, so we may assume $U \not\perp V$.   Choose $(x_U, x_V)$, $(y_U, y_V)$, $(z_U, z_V) \in \calR_{UV}$ and set $m_U = \mu_U(x_U, y_U, z_U)$, $m_V = \mu_V(x_V, y_V, z_V)$. By Lemma \ref{lem:intervals}, we deduce $m_V \in [x_V, y_V]_{C + \rho C} \subseteq [x_V, y_V]_B$, hence $[x_V, m_V]_B \subseteq [x_V, y_V]_{L}$ for some $L = L(C,B,\rho)$. If $O^U_V \in [x_V, y_V]_L \cap [y_V, z_V]_L \cap [z_V, x_V]_L$ then by Lemma \ref{lem:tripod}, we deduce $O^U_V \approx_R m_V$ for some $R \geq 0$ depending on $L$, $C$, and $\rho$, hence $(m_U, m_V) \approx_R (m_U, O^U_V) \in \calR_{UV}$. We may thus assume that $O^U_V \notin [x_V, y_V]_{L} \supseteq [x_V, m_V]_{B}$ (up to cyclically permuting $x_V, y_V, z_V$). By the bounded coarse interval image assumption, we deduce $\theta^V_U m_V \approx_B \theta^V_U x_V \approx_B \theta^V_U y_V$. Since $(x_U, x_V)$, $(y_U, y_V)$ belong to $\calR_{UV}$, it follows that $x_U \approx_K \theta^V_U x_V \approx_B \theta^V_U y_V \approx_K y_U$, and so $m_U \approx_{\rho(2K + B)} x_U$. Consequently, $m_U \approx_{R} \theta^V_U m_V$ for $R \geq \rho(2K + B) + K + B$. Therefore $(m_U, m_V) \approx_R (\theta^V_U m_V, m_V) \in \calR_{UV}$.
  \endproof

  Consequently, by Lemma \ref{lem:pairwise-median}, each $\calR_{UV}$ admits a coarse median induced by $\mu_U \times \mu_V$; moreover, the coarse median constants and upper controls can be chosen uniformly. Following \cite[Section 6.1]{Tang-rips}, we may define a \uc~coarse median diagram $M \from \sJ \sqto \CMed$ from the family $\{(\C U, \mu_U)\}_{U \in\calS}$ equipped with the pairwise constraints $\calR_{UV}$ as follows. Each vertex of $\sJ$ is labelled by either an element $U \in\calS$ or by a distinct (unordered) pair $UV$; the directed arrows are given by $U \leftarrow {UV} \rightarrow V$ for each $UV$. The objects of $M$ are the $(\C U, \mu_U)$ for each $U \in \calS$, together with the $(\calR_{UV}, \mu_U \times \mu_V)$ for each (unordered) pair $U,V \in \calS$ distinct; the bonding maps are given by the (1--Lipschitz) factorwise projections $\C U \leftarrow \calR_{UV} \rightarrow \C V$. Note that these bonding maps are uniformly CMP.

  Let us now apply this setup to HHSs. Here, we are given a quasigeodesic space $\X$, a family $\{\C U\}_{U \in\calS}$ of $\delta$--hyperbolic spaces, for some $\delta \geq 0$, together with a uniformly coarsely Lipschitz family of projections $\{\lambda_U \from \X \to \C U\}_{U \in\calS}$. The index set $\calS$ is equipped with an orthogonality relation $\perp$. The HHS Axioms assert that pairwise projections $\lambda_U \times \lambda_V$ factor through some pairwise constraint $\calR_{UV} \subseteq \C U \times \C V$ up to uniform error: for $U \perp V$ there is no constraint, for $U \not\perp V$ they satisfy either the Behrstock Inequality or the Bounded Geodesic Image property. Since the spaces $\C U$ are uniformly hyperbolic, we may endow them with uniform coarse medians $\mu_U$; explicitly, let $\mu_U(x_U, y_U, z_U)$ be any $\delta$--centre of a geodesic triangle with corners $x_U, y_U, z_U \in \C U$. Each $\calR_{UV}$ either equals $\C U \times \C V$ or satisfies the Bounded Coarse Interval condition with uniform error and so, by Lemma \ref{lem:pairwise-median}, they admit uniform coarse medians induced by $\mu_U \times \mu_V$. We thus obtain a \uc~coarse median diagram $M \from \sJ \sqto \CMed$ of the form given above.

  \begin{corollary}
   Let $\X$ be an HHS and $M \from \sJ \sqto \CMed$ be the corresponding \uc~coarse median diagram as given above. Then there exists a coarse median $\nu$ on $\X$, unique up to closeness, such that $\lambda \from (\X, \nu) \sqtto M$ is a universal quasigeodesic \uc~coarse median cone.
  \end{corollary}

  \proof
  By \cite[Theorem 6.7]{Tang-rips}, there is a universal quasigeodesic \uc~cone $\lambda \from \X \sqtto \hat M$ whose legs are given by the $\lambda_U$ and $\lambda_U \times \lambda_V$. The result follows using Theorem \ref{thm:uqc-median}.
  \endproof

\providecommand{\bysame}{\leavevmode\hbox to3em{\hrulefill}\thinspace}
\providecommand{\MR}{\relax\ifhmode\unskip\space\fi MR }
\providecommand{\MRhref}[2]{
\href{http://www.ams.org/mathscinet-getitem?mr=#1}{#2}
}
\providecommand{\href}[2]{#2}

\end{document}